\documentclass[10pt]{amsart}
\usepackage{fullpage,amssymb,amsmath,hyperref,color,enumitem,caption,cite,comment}
\usepackage[abs]{overpic}
\pdfoutput=1

\newcommand{\bbC}{\mathbb{C}}

\newcommand{\bbN}{\mathbb{N}}

\newcommand{\bbQ}{\mathbb{Q}}

\newcommand{\bbZ}{\mathbb{Z}}

\newcommand{\calK}{\mathcal{K}}

\newcommand{\calO}{\mathcal{O}}

\newcommand{\OF}{\calO_F}

\newcommand{\QQbar}{\overline{\bbQ}}

\newcommand{\Zbar}{\overline{\bbZ}}

\renewcommand{\hbar}{\overline{h}}

\newcommand{\frakp}{\mathfrak{p}}

\newcommand{\ord}{\operatorname{ord}}

\newcommand{\Spec}{\operatorname{Spec}}

\newcommand{\Res}{\operatorname{Res}}

\renewcommand{\tilde}{\widetilde}
\renewcommand{\phi}{\varphi}

\newtheorem{thm}{Theorem}[section]
\newtheorem{lem}[thm]{Lemma}
\newtheorem{prop}[thm]{Proposition}

\theoremstyle{definition}
\newtheorem{ques}{Question}
\theoremstyle{remark}
\newtheorem*{rem}{Remark}

\numberwithin{equation}{section}

\title{Preperiodic portraits for unicritical polynomials}
\author{John R. Doyle}
\address{Department of Mathematics \\
University of Rochester \\
Rochester, NY 14627} 
\email{john.doyle@rochester.edu}

\begin{document}

\begin{abstract}
Let $K$ be an algebraically closed field of characteristic zero, and for $c \in K$ and an integer $d \ge 2$, define $f_{d,c}(z) := z^d + c \in K[z]$. We consider the following question: If we fix $x \in K$ and integers $M \ge 0$, $N \ge 1$, and $d \ge 2$, does there exist $c \in K$ such that, under iteration by $f_{d,c}$, the point $x$ enters into an $N$-cycle after precisely $M$ steps? We conclude that the answer is generally affirmative, and we explicitly give all counterexamples. When $d = 2$, this answers a question posed by Ghioca, Nguyen, and Tucker.
\end{abstract}



\maketitle

\section{Introduction}\label{sec:intro}

Throughout this article, $K$ will be an algebraically closed field of characteristic zero. Let $\phi(z) \in K[z]$ be a polynomial of degree $d \ge 2$. For $n \ge 0$, let $\phi^n$ denote the $n$-fold composition of $\phi$; that is, $\phi^0$ is the identity map, and $\phi^n = \phi \circ \phi^{n-1}$ for each $n \ge 1$. A point $x \in K$ is \textbf{preperiodic} for $\phi$ if there exist integers $M \ge 0$ and $N \ge 1$ for which $\phi^{M+N}(x) = \phi^M(x)$. In this case, the minimal such $M$ is called the \textbf{preperiod} of $x$, and the minimal such $N$ is called the \textbf{eventual period} of $x$. If the preperiod $M$ is zero, then we say that $x$ is \textbf{periodic} of \textbf{period} $N$. If $M \ge 1$, then we call $x$ \textbf{strictly preperiodic}. If $M$ and $N$ are the preperiod and period, respectively, then we call the pair $(M,N)$ the \textbf{preperiodic portrait} (or simply \textbf{portrait}) of $x$ under $\phi$.

A natural question to ask is the following:

\begin{ques}\label{ques:baker}
Given a polynomial $\phi \in K[z]$ of degree at least 2, and given integers $M \ge 0$ and $N \ge 1$, does there exist an element $x \in K$ with portrait $(M,N)$ for $\phi$?
\end{ques}

This question was completely answered by Baker \cite{baker:1964} in the case that $M = 0$. (See also \cite[Thm. 1]{kisaka:1995} for the corresponding statement for rational functions.) Before stating Baker's result, though, we give an example of a polynomial that fails to admit points with a certain portrait.

Consider the polynomial $\phi(z) = z^2 - 3/4$. A quadratic polynomial $z^2 + c$ typically admits two points of period two, forming a single two-cycle; however, the polynomial $\phi$ admits no such points. Indeed, such a point $x$ would satisfy $\phi^2(x) = x$, but one can see that
	\[ \phi^2(z) - z = (z - 3/2)(z + 1/2)^3, \]
and each of the points $3/2$ and $-1/2$ is actually a \emph{fixed point} for $\phi$. This example stems from the fact that $c = -3/4$ is the root of the period-2 hyperbolic component of the Mandelbrot set. In other words, $c = -3/4$ is a \emph{bifurcation point} --- it is the parameter at which the two points forming a two-cycle for $z^2 + c$ merge into one point, effectively collapsing the two-cycle to a single fixed point. To illustrate this, we let $Y$ be the affine curve defined by $(X^2 + C)^2 + C - X = 0$. For a given $c \in K$, if $x$ is a fixed point or a point of period 2 for $z^2 + c$, then $(x,c) \in Y(K)$. This suggests a natural decomposition of $Y$ into two irreducible components --- a ``period 1 curve" $Y_1$, defined by $X^2 + C - X = 0$, and a ``period 2 curve" $Y_2$, defined by $\frac{(X^2 + C)^2 + C - X}{X^2 + C - X} = X^2 + X + C + 1 = 0$, illustrated in Figure~\ref{fig:curves}. The bifurcation at $c = -3/4$ may be seen by letting $c$ tend to $-3/4$ and observing that the two points on $Y_2$ lying over $c$ (corresponding to the two points of period 2 for $z^2 + c$) approach a single point on $Y_1$ (corresponding to a \emph{fixed point} for $z^2 - 3/4$).

\begin{figure}
\begin{overpic}[scale=.5]{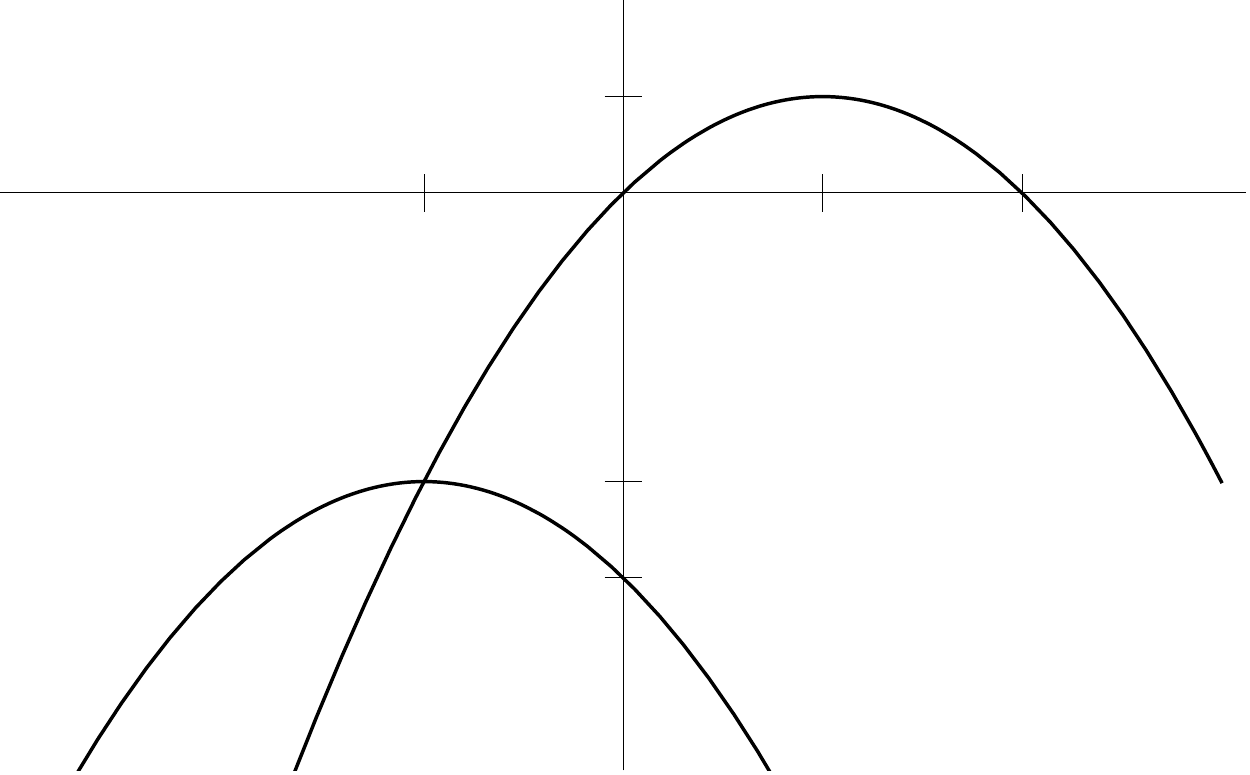}
\put(95,25){\Small $-1$}
\put(95,40){\Small $-3/4$}
\put(48,70){\Small $-1/2$}
\put(70,94){\Small $1/4$}
\put(113,70){\Small $1/2$}
\put(146,70){\Small $1$}
\put(127,102){$Y_1$}
\put(15,30){$Y_2$}
\put(183,80){$X$}
\put(89,116){$C$}
\end{overpic}
\caption{The affine curve $Y: (X^2 + C)^2 + C - X = 0$, with components $Y_1$ and $Y_2$.}
\label{fig:curves}
\end{figure}
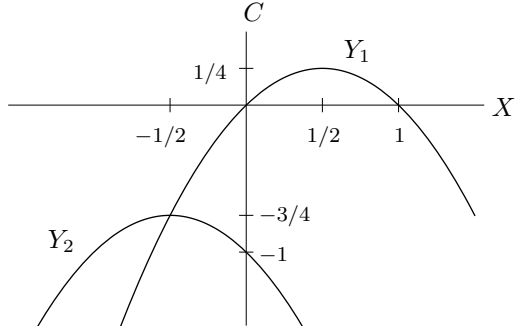

Baker showed that the polynomial $\phi(z) = z^2 - 3/4$ is, in some sense, the \emph{only} polynomial of degree at least 2 that fails to admit points of a given period. To make this more precise, we first recall the following terminology and notation: two polynomials $\phi,\psi \in K[z]$ are \textbf{linearly conjugate} if there exists a linear polynomial $\ell(z) = az + b$ such that $\psi = \ell^{-1} \circ \phi \circ \ell$, and in this case we write $\phi \sim \psi$. Note that $\psi^n = \ell^{-1} \circ \phi^n \circ \ell$, so this relation is the appropriate notion of equivalence in dynamics. In particular, $x \in K$ has portrait $(M,N)$ for $\psi$ if and only if $\ell(x)$ has portrait $(M,N)$ for $\phi$.

\begin{thm}[{Baker \cite[Thm. 2]{baker:1964}}]\label{thm:baker}
Let $\phi(z) \in K[z]$ with degree $d \ge 2$, and let $N \ge 1$ be an integer. If $\phi(z) \not \sim z^2 - 3/4$, then $\phi$ admits a point of period $N$. If $\phi(z) \sim z^2 - 3/4$, then $\phi$ admits a point of period $N$ if and only if $N \ne 2$.
\end{thm}

Though Baker was only considering \emph{periodic} points, and therefore only answered Question~\ref{ques:baker} for $M = 0$, it is not difficult to extend his result to the case $M > 0$.

\begin{prop}\label{prop:baker}
Let $\phi(z) \in K[z]$ with degree $d \ge 2$, and let $M \ge 0$ and $N \ge 1$ be integers. If $\phi(z) \not \sim z^2 - 3/4$, then $\phi$ admits a point of portrait $(M,N)$. If $\phi(z) \sim z^2 - 3/4$, then $\phi$ admits a point of portrait $(M,N)$ if and only if $N \ne 2$.
\end{prop}

\begin{proof}
The claim that if $\phi(z) \sim z^2 - 3/4$, then $\phi$ does not admit points of portrait $(M,2)$ follows immediately from Theorem~\ref{thm:baker}. We now suppose either that $\phi \not \sim z^2 - 3/4$, or that $\phi \sim z^2 - 3/4$ and $N \ne 2$, and we show that there exists a point of portrait $(M,N)$ for $\phi$.

The $M = 0$ case is precisely Theorem~\ref{thm:baker}, and the $M = 1$ case follows from the fact (see \cite[Lem. 4.24]{ghioca/nguyen/tucker:2015}) that if a polynomial admits a point of period $N$, then it also admits a point of portrait $(1,N)$. Now suppose $M \ge 2$. By induction, there exists $y \in K$ with portrait $(M-1,N)$ for $\phi$. Since $y$ is itself strictly preperiodic, it is easy to see that any preimage $x$ of $y$ has portrait $(M,N)$.
\end{proof}

We now consider the dual question to Question~\ref{ques:baker}: given an element $x \in K$, and given integers $M \ge 0$ and $N \ge 1$, does there exist a polynomial $\phi(z) \in K[z]$ of degree at least 2 for which $x$ has portrait $(M,N)$?

It is not difficult to see that the answer to this question is ``yes." Let $\psi(z)$ be any polynomial not linearly conjugate to $z^2 - 3/4$, so that $\psi$ is guaranteed to admit a point $\zeta$ of portrait $(M,N)$ by Proposition~\ref{prop:baker}. If we let $\ell(z) := z + (\zeta - x)$, so that $\ell(x) = \zeta$, then $x$ has portrait $(M,N)$ for $\phi := \ell^{-1} \circ \psi \circ \ell$.

This suggests that an appropriate dual question should only allow us to consider one polynomial (or, at worst, finitely many) from each linear conjugacy class. Also, since we are imposing a single condition on the polynomial $\phi$ --- namely, that the given point $x$ have portrait $(M,N)$ under $\phi$ --- we ought to consider a one-parameter family of maps for each degree $d \ge 2$.

This naturally leads us to consider the class of \emph{unicritical polynomials}; i.e., polynomials with a single (finite) critical point. Every unicritical polynomial is linearly conjugate to a polynomial of the form
	\[ f_{d,c}(z) := z^d + c, \]
so we consider only this one-parameter family of polynomials for each $d \ge 2$. Note that $f_{d,c} \sim f_{d,c'}$ if and only if $c/c'$ is a $(d-1)$th root of unity, so this family contains only finitely many polynomials from a given conjugacy class.

We now ask the following more restrictive question:

\begin{ques}\label{ques:main}
Given $(x,M,N,d) \in K \times \bbZ^3$ with $M \ge 0$, $N \ge 1$, and $d \ge 2$, does there exist $c \in K$ for which $x$ has portrait $(M,N)$ under $f_{d,c}$?
\end{ques}

If there does exist such an element $c \in K$, we will say that $x$ \textbf{realizes portrait $(M,N)$ in degree $d$}. Before stating our main result, we give some examples of tuples $(x,M,N,d)$ for which the answer to Question~\ref{ques:main} is negative.

First, observe that $0$ cannot realize portrait $(1,N)$ in degree $d$ for any $N \ge 1$ and $d \ge 2$. Indeed, suppose $f_{d,c}(0) = c$ is periodic of period $N$, which is equivalent to saying that $0$ either has portrait $(1,N)$ or is periodic of period $N$ itself. Since a periodic point must have precisely one periodic preimage, and since the only preimage of $c$ under $f_{d,c}$ is $0$, we must have that $0$ is periodic of period $N$. This particular counterexample is special to unicritical polynomials, since the failure of $0$ to realize portrait $(1,N)$ is due to the fact that $f_{d,c}$ is totally ramified at $0$ for all $c \in K$, as illustrated in Figure~\ref{fig:totally_ramified}.

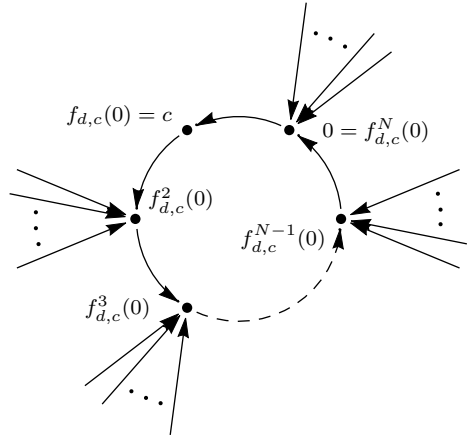
\begin{figure}
\begin{overpic}[scale=.5]{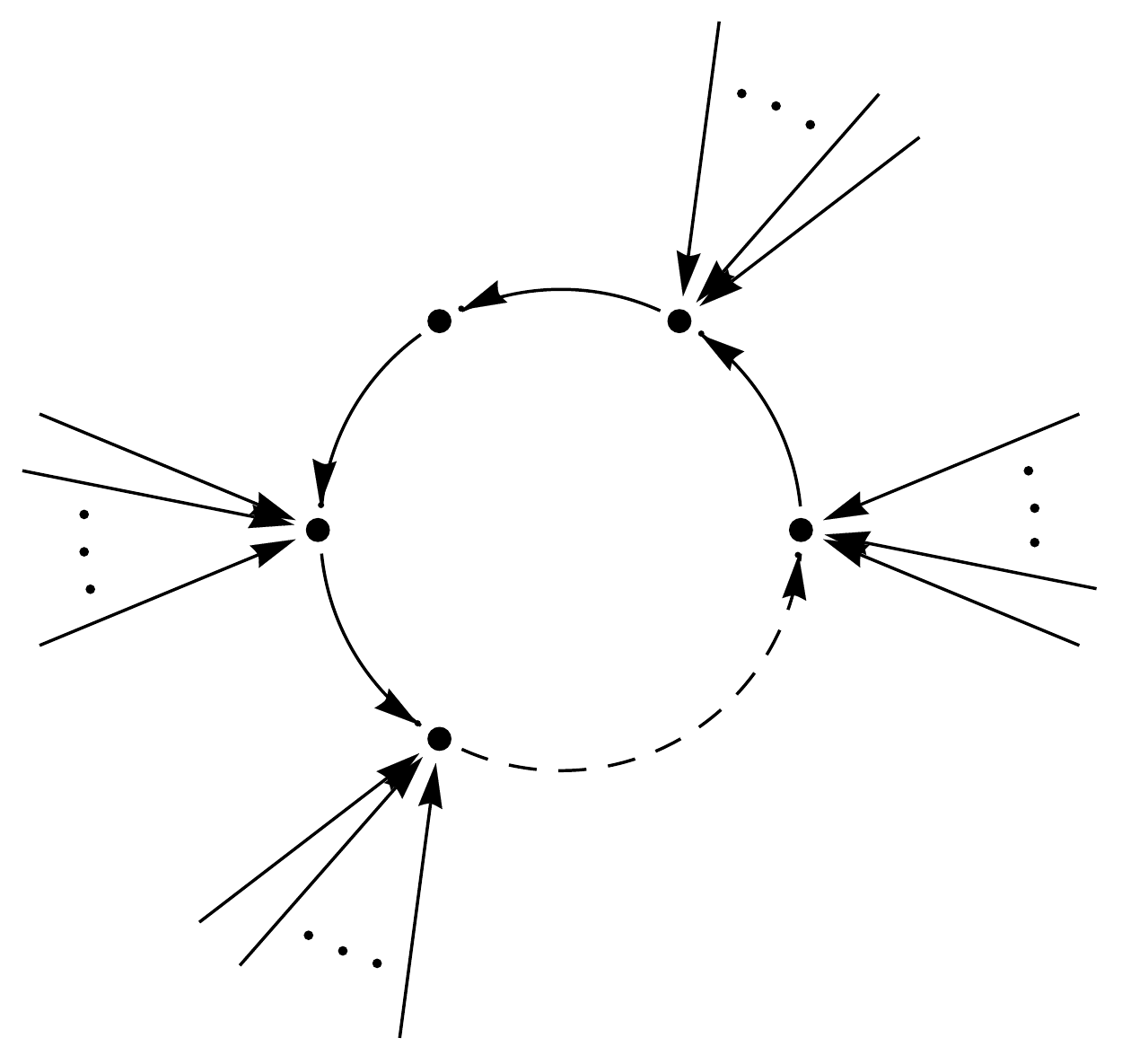}
\put(122,117){\Small $0 = f_{d,c}^N(0)$}
\put(25,123){\Small $f_{d,c}(0) = c$}
\put(56,91){\Small $f_{d,c}^2(0)$}
\put(32,50){\Small $f_{d,c}^3(0)$}
\put(91,76){\Small $f_{d,c}^{N-1}(0)$}
\end{overpic}
\caption{If $c$ is periodic for $f_{d,c}$, then $c$ has no strictly preperiodic preimage}
\label{fig:totally_ramified}
\end{figure}

Next, consider $x = -1/2$. We show that $x$ cannot realize portrait $(0,2)$ in degree 2; that is, there is no $c \in K$ such that $-1/2$ has period 2 for $f_{2,c}$. If there were such a parameter $c$, then we would have
	\[ 0 = f_{2,c}^2(-1/2) - (-1/2) = (c + 3/4)^2. \]
However, if we take $c = -3/4$, then $-1/2$ is a fixed point for $f_{2,c}$. There is therefore no $c \in K$ such that $-1/2$ has period 2 for $f_{2,c}$. This is illustrated in Figure~\ref{fig:curves}, which shows that the only point on the ``period 2 curve" $Y_2$ lying over $x = -1/2$ also lies on the ``period 1 curve" $Y_1$.

An argument similar to the one in the preceding paragraph shows that $x = 1/2$ cannot realize portrait $(1,2)$ in degree 2 and that $x = \pm 1$ cannot realize portrait $(2,2)$ in degree 2. Our main result states that these are the only instances where the answer to Question~\ref{ques:main} is negative.

\begin{thm}\label{thm:main}
Let $K$ be an algebraically closed field of characteristic zero, and let $(x,M,N,d) \in K \times \bbZ^3$ with $M \ge 0$, $N \ge 1$, and $d \ge 2$. Then there exists $c \in K$ for which $x$ has portrait $(M,N)$ under $f_{d,c}$ if and only if
	\[
		(x,M) \ne (0,1) \mbox{ and } (x,M,N,d) \not \in \left\{ \left(-\frac{1}{2},0,2,2 \right), \left(\frac{1}{2}, 1, 2, 2 \right), \left( \pm 1, 2, 2, 2 \right) \right\}.
	\]
\end{thm}

Ghioca, Nguyen, and Tucker \cite{ghioca/nguyen/tucker:2015} consider the more general problem of ``simultaneous multi-portraits" for polynomial maps: Given a $(d-1)$-tuple of points $(x_1,\ldots,x_{d-1}) \in K^{d-1}$, and given $(d-1)$ portraits $(M_1,N_1),\ldots,(M_{d-1},N_{d-1})$, does there exist a degree $d$ polynomial in standard form
	\[
		\phi(z) := z^d + c_{d-2}z^{d-2} + \cdots + c_1z + c_0
	\]
such that, for each $i \in \{1,\ldots,d-1\}$, $x_i$ has portrait $(M_i,N_i)$ for $\phi$? The $d = 2$ case of this question is precisely the $d = 2$ case of Question~\ref{ques:main} in the present article. In an earlier version of their article, the authors of \cite{ghioca/nguyen/tucker:2015} provided $(x,M) = (0,1)$ and $(x,M,N) = (-1/2,0,2)$ as examples of the failure of a given point to realize a given portrait in degree $2$, and they asked whether there were any other such failures. Theorem~\ref{thm:main} completely answers this question.

The main tool used in \cite{ghioca/nguyen/tucker:2015} to approach the multi-portrait problem is a result for \emph{single} portraits, which they are then able to extend to multi-portraits by an iterative process. Their main result (\cite[Thm. 1.3]{ghioca/nguyen/tucker:2015}), when applied to the case of unicritical polynomials, says the following: For a fixed $d \ge 2$, if $(x,M) \ne (0,1)$ and if $(M,N)$ avoids an effectively computable finite subset of $\bbZ_{\ge 0} \times \bbN$, then every $x \in K$ realizes portrait $(M,N)$ in degree $d$. One might therefore be able to use the techniques of \cite{ghioca/nguyen/tucker:2015}, involving Diophantine approximation, to prove Theorem~\ref{thm:main} for fixed values of $d$. In this article, however, we take an entirely different approach by using properties of certain algebraic curves, which we call \emph{dynamical modular curves}, that are defined in terms of the dynamics of the maps $f_{d,c}$.

We now briefly outline the rest of this article. In \textsection \ref{sec:dyn_curves}, we record a number of known properties of dynatomic polynomials and the corresponding dynamical modular curves. Section \ref{sec:main_thm} contains the proof of Theorem~\ref{thm:main}, which is generally based on the following principle: For each portrait $(M,N)$, there is a curve $Y_1(M,N)$ whose points parametrize maps $f_{d,c}$ together with points of portrait $(M,N)$. If $x$ does not achieve portrait $(M,N)$ in degree $d$, then each point on $Y_1(M,N)$ lying above $x$ must also lie on $Y_1(m,n)$ for some integers $m \le M$ and $n \le N$ with $m < M$ or $n < N$. Once the degree of $Y_1(M,N)$ becomes large enough, however, this proves to be impossible (excluding the special case $(x,M) = (0,1)$, discussed above). In the final section, we discuss some related open problems.

\subsection*{Acknowledgments}

I would like to thank Tom Tucker for introducing me to this problem, as well as for several useful comments and suggestions on an earlier draft of this paper.

\section{Dynatomic polynomials and dynamical modular curves}\label{sec:dyn_curves}

\subsection{Dynatomic polynomials}\label{subsec:dyn_poly}

If $c$ is an element of $K$ and $x \in K$ is a point of period $N$ for $f_{d,c}$, then $(x,c)$ is a solution to the equation $f_{d,C}^N(X) - X = 0$. However, $(x,c)$ is also a solution to this equation whenever $x$ is a point of period dividing $N$ for $f_{d,c}$. We therefore define the \textbf{$N$th dynatomic polynomial} to be the polynomial
	\[
		\Phi_N(X,C) := \prod_{n \mid N} \left(f_{d,C}^n(X) - X\right)^{\mu(N/n)} \in \bbZ[X,C]
	\]
(where $\mu$ is the M\"{o}bius function), which has the property that
	\begin{equation}\label{eq:factorization}
		f_{d,C}^N(X) - X = \prod_{n \mid N} \Phi_n(X,C)
	\end{equation}
for all $N \in \bbN$ --- see \cite[p. 571]{morton/vivaldi:1995}. For simplicity of notation, we omit the dependence on $d$. The fact that $\Phi_N(X,C)$ is a polynomial is shown in \cite[Thm. 4.5]{silverman:2007}, and it is not difficult to see that $\Phi_N$ is monic in both $X$ and $C$.

If $(x,c) \in K^2$ is such that $\Phi_N(x,c) = 0$, we say that $x$ has \textbf{formal period} $N$ for $f_{d,c}$. Every point of exact period $N$ has formal period $N$, but in some cases a point of formal period $N$ may have exact period $n$ a proper divisor of $N$. If $x$ is such a point, then $x$ appears in the cycle $\{x,f_{d,c}(x),\ldots,f_{d,c}^{N-1}(x)\}$ with multiplicity $N/n$, and this multiplicity is captured by $\Phi_N$. In particular, $x$ is a multiple root of the polynomial $\Phi_N(X,c) \in K[X]$, so we have the following:

\begin{lem}\label{lem:smaller_period}
Let $c \in K$. Suppose that $x \in K$ has formal period $N$ and exact period $n < N$ for $f_{d,c}$. Then
	\[
		\left. \frac{\partial \Phi_N(X,C)}{\partial X} \right|_{(x,c)} = 0.
	\]
Moreover, $x,c \in \QQbar \setminus \Zbar$.
\end{lem}

For an illustration of this phenomenon, see Figure~\ref{fig:curves}, which shows the curves $Y_1 : \Phi_1(X,C) = 0$ and $Y_2 : \Phi_2(X,C) = 0$ in the degree $d = 2$ case. One can see in the figure that the $X$-partial of $\Phi_2(X,C)$ vanishes at the point $(x,c) = (-1/2,-3/4)$ on $Y_2$, where $x$ actually has period 1 for $f_{2,c}$.

We also briefly explain the statement that $x,c \in \QQbar \setminus \Zbar$. (See also \cite[p. 582]{morton/vivaldi:1995}.) Since $x$ has formal period $N$ and exact period $n < N$ for $f_{d,c}$, $c$ is a root of the resultant
	\[
		\Res_X\left(\Phi_N(X,C),\Phi_n(X,C) \right) \in \bbZ[C].
	\]
Thus $c \in \QQbar$, hence also $x \in \QQbar$ since $\Phi_N(x,c) = 0$. On the other hand, a multiple root $x$ of the polynomial $\Phi_N(X,c) \in K[X]$ must also be a multiple root of $f_{d,c}^N(X) - X$, so
	\begin{equation}\label{eq:nonint}
		0 = (f_{d,c}^N)'(x) - 1 = d^N \prod_{k=0}^{N-1} f_{d,c}^k(x) - 1.
	\end{equation}
Since $c \in \Zbar$ if and only if $x \in \Zbar$ ($\Phi_N$ is monic in both variables), and since the rightmost expression of \eqref{eq:nonint} cannot vanish if $x,c \in \Zbar$ (the expression is congruent to $-1$ modulo $d\Zbar$), we must have $x,c \not \in \Zbar$.

Finally, for an application in \textsection \ref{sec:smaller_period}, we compare the degree of $\Phi_N$ to the degrees of the polynomials $\Phi_n$ with $n$ properly dividing $N$. Let
	\[
		D(N) := \sum_{n \mid N} \mu(N/n)\cdot d^n
	\]
denote the degree of $\Phi_N$ in $X$. Note that $\Phi_N$ has degree $D(N)/d$ in $C$.

\begin{lem}\label{lem:d(N)}
Let $N \in \bbN$ be a positive integer. Then
	\[ D(N) > \sum_{\substack{n \mid N\\n < N}} D(n), \]
\emph{unless} $N = d = 2$, in which case equality holds.
\end{lem}

\begin{proof}
If $N = 1$, then the statement is trivial. We therefore assume $N \ge 2$.

Since the polynomial $f_{d,C}^N(X) - X$ has degree $d^N$ in $X$, we can see from \eqref{eq:factorization} that the sum appearing in the lemma is actually equal to $d^N - D(N)$. Hence, it suffices to prove the equivalent inequality
	\begin{equation}\label{eq:simpler_ineq}
		D(N) > \frac{1}{2} \cdot d^N.
	\end{equation}

We first obtain a rough lower bound for $D(N)$, using the fact that the largest proper divisor of $N$ has size at most $\lfloor N/2 \rfloor$:
\[
	D(N)
			= \sum_{n \mid N} \mu(N/n) d^n \ge d^N - \sum_{\substack{n \mid N \\ n < N}} d^n
			\ge d^N - \sum_{n=1}^{\lfloor N/2 \rfloor} d^n
			= d^N - \frac{d}{d-1} \cdot \left(d^{\lfloor N/2 \rfloor} - 1\right)
			> d^N - \frac{d}{d-1} \cdot d^{N/2}.
\]
It therefore suffices to show that
	\[
		\frac{d}{d-1} \cdot d^{N/2} \le \frac{1}{2} \cdot d^N,
	\]
which we can rearrange to become
	\begin{equation}\label{eq:even_simpler_ineq}
		d^{N/2 - 1} \ge \frac{2}{d-1}.
	\end{equation}
First, suppose $d = 2$. Then \eqref{eq:even_simpler_ineq} becomes
	\[ 2^{N/2 - 1} \ge 2, \]
which is satisfied for $N \ge 4$. For $N = 2$, we have $D(2) = 2 = D(1)$, which gives us the desired equality in this case. For $N = 3$, we have $D(3) = 6 > D(1)$.

Finally, when $d \ge 3$, we observe that the right hand side of \eqref{eq:even_simpler_ineq} is at most $1$, while the left hand side is at least $1$ when $N \ge 2$. Therefore \eqref{eq:even_simpler_ineq} is satisfied whenever $d \ge 3$ and $N \ge 2$, completing the proof.
\end{proof}

\subsection{Generalized dynatomic polynomials}

To say that a point $x \in K$ has portrait $(M,N)$ for $f_{d,c}$ is to say that $f_{d,c}^M(x)$ has period $N$ but $f_{d,c}^{M-1}(x)$ does not. For this reason, if $M$ and $N$ are positive integers, we define the \textbf{generalized dynatomic polynomial} $\Phi_{M,N}(X,C)$ to be the polynomial
	\begin{equation}\label{eq:gen_dyn}
		\Phi_{M,N}(X,C) := \frac{\Phi_N(f_{d,C}^M(X),C)}{\Phi_N(f_{d,C}^{M-1}(X),C)} \in \bbZ[X,C].
	\end{equation}
For convenience, we set $\Phi_{0,N} := \Phi_N$, and we again omit the dependence on $d$. That $\Phi_{M,N}$ is a polynomial is shown in \cite[Thm. 1]{hutz:2015}. If $(x,c) \in K^2$ satisfies $\Phi_{M,N}(x,c) = 0$, we will say that $x$ has \textbf{formal portrait} $(M,N)$ for $f_{d,c}$, and we similarly attach ``formal" to the terms ``preperiod" and ``eventual period" in this case. As in the periodic case, every point with exact portrait $(M,N)$ has formal portrait $(M,N)$, but a point with formal portrait $(M,N)$ may have exact portrait $(m,n)$ with $m < M$ or $n$ a proper divisor of $N$. It is again not difficult to see that $\Phi_{M,N}$ is monic in both $X$ and $C$, and that, when $M \ge 1$, $\Phi_{M,N}$ has degree $(d-1) d^{M-1}  D(N)$ in $X$ and degree $(d - 1)d^{M-2}D(N)$ in $C$.

Let $Y_1(M,N)$ denote the affine plane curve defined by $\Phi_{M,N}(X,C) = 0$. We call a curve defined in this way a \textbf{dynamical modular curve}. We summarize the relevant properties of $Y_1(M,N)$ in the following lemma:

\begin{lem}\label{lem:factors}
Let $K$ be an algebraically closed field of characteristic zero, and let $M \ge 0$ and $N \ge 1$ be integers.
\begin{enumerate}
	\item If $M = 0$, then the curve $Y_1(0,N)$ is nonsingular and irreducible over $K$.
	\item If $M \ge 1$, then for each $d$th root of unity $\zeta$, define
	\begin{equation}\label{eq:Psi}
		\Psi_{M,N}^{\zeta}(X,C) := \Phi_N(\zeta f_{d,C}^{M-1}(X),C).
	\end{equation}	
Then
	\begin{equation}\label{eq:factors}
		\Phi_{M,N}(X,C) = \prod_{\substack{\zeta^d = 1 \\ \zeta \ne 1}} \Psi_{M,N}^{\zeta}(X,C).
	\end{equation}
Each of the polynomials $\Psi_{M,N}^{\zeta}(X,C)$ is irreducible over $K$, so $Y_1(M,N)$ has exactly $(d-1)$ irreducible components. Each of the components is smooth, and the points of intersection of the components are precisely those points $(x,c)$ with $f_{d,c}^{M-1}(x) = 0$.
\end{enumerate}
\end{lem}

Part (A) was originally proven in the $d = 2$ case by Douady and Hubbard (smoothness; \cite[\textsection XIV]{douady/hubbard:1985}), and Bousch (irreducibility; \cite[Thm. 1 (\textsection 3)]{bousch:1992}). A subsequent proof of (A) in the $d = 2$ case was later given by Buff and Lei \cite[Thm. 3.1]{buff/lei:2014}. For $d \ge 2$, irreducibility was proven by Lau and Schleicher \cite[Thm. 4.1]{lau/schleicher:1994} using analytic methods and by Morton \cite[Cor. 2]{morton:1996} using algebraic methods, and both irreducibility and smoothness were later proven by Gao and Ou \cite[Thms. 1.1, 1.2]{gao/ou} using the methods of Buff-Lei. Part (B) is due to Gao \cite[Thm. 1.2]{gao}. The lemma was originally proven over $\bbC$, but the Lefschetz principle allows us to extend the result to arbitrary fields of characteristic zero: since the curves $Y_1(M,N)$ are all defined over $\bbZ$, any singular points and irreducible components would be defined over a finitely generated extension of $\bbQ$, which could then be embedded into $\bbC$.

Finally, we briefly explain the factorization in \eqref{eq:factors}. If $x$ has portrait $(M,N)$ for $f_{d,c}$, then $f_{d,c}^M(x)$ is periodic of period $N$, so precisely one preimage of $f_{d,c}^M(x)$ is also periodic. The periodic preimage cannot be $f_{d,c}^{M-1}(x)$, since this would imply that $x$ has portrait $(m,N)$ for some $m \le M-1$. Since any two preimages of a given point under $f_{d,c}$ differ by a $d$th root of unity, this implies that $\zeta f_{d,c}^{M-1}(x)$ is periodic for some $d$th root of unity $\zeta \ne 1$, and therefore $\Psi^\zeta_{M,N}(x,c) = 0$ for that particular value of $\zeta$.

\section{Formal portraits and exact portraits}\label{sec:main_thm}

In order to prove Theorem~\ref{thm:main}, we must describe those conditions under which a point may have formal portrait different from its exact portrait under the map $f_{d,c}$. We begin by giving a necessary and sufficient condition for the exact preperiod of a point to be strictly less than its formal preperiod.

\begin{lem}\label{lem:smaller_preperiod}
Let $M,N \in \bbN$, and suppose $x$ has formal portrait $(M,N)$ for $f_{d,c}$. Then $x$ has exact preperiod strictly less than $M$ if and only if $f_{d,c}^{M-1}(x) = 0$. In this case, both $x$ and $c$ are algebraic integers and $0$ is periodic of period equal to $N$ (hence $x$ has eventual period $N$).
\end{lem}

\begin{proof}
First, suppose $x$ has exact preperiod $m < M$ for $f_{d,c}$. By Lemma~\ref{lem:factors}, since $x$ has formal portrait $(M,N)$ for $f_{d,c}$, we must have
	\[ \Phi_N(\zeta f_{d,c}^{M-1}(x),c) = 0 \]
for some $d$th root of unity $\zeta \ne 1$. Hence $\zeta f_{d,c}^{M-1}(x)$ is periodic. On the other hand, $f_{d,c}^{M-1}(x)$ is also periodic, since $f_{d,c}^m(x)$ is periodic and $m \le M - 1$. Both $\zeta f_{d,c}^{M-1}(x)$ and $f_{d,c}^{M-1}(x)$ are preimages of $f_{d,c}^M(x)$; since a point can only have a single periodic preimage, it follows that $\zeta f_{d,c}^{M-1}(x) = f_{d,c}^{M-1}(x)$, which then implies that $f_{d,c}^{M-1}(x) = 0$.

Conversely, suppose that $f_{d,c}^{M-1}(x) = 0$. Since $x$ has formal portrait $(M,N)$ for $f_{d,c}$, the factorization in Lemma~\ref{lem:factors} implies that $\Phi_N(0,c) = 0$, so $0$ is periodic for $f_{d,c}$. In particular, this means that the preperiod of $x$ is at most $M - 1$.

The fact that $\Phi_N(0,c) = 0$ implies that $0$ is periodic for $f_{d,c}$ and, since $\Phi_N(0,C)$ is monic in $C$, that $c \in \Zbar$. Moreover, since $f_{d,c}^{M-1}(x) = 0$, we also conclude that $x \in \Zbar$. The final claim --- that the period of $0$ (and hence the eventual period of $x$) is equal to $N$ --- follows from Lemma~\ref{lem:smaller_period}.
\end{proof}

As a consequence of Lemma~\ref{lem:smaller_preperiod}, we see that if $x$ has formal portrait $(M,N)$ and exact portrait $(m,n)$ for $f_{d,c}$, then either $m = M$ or $n = N$. We can actually say a bit more, using the fact that if $x$ is preperiodic for $f_{d,c}$ --- which is necessarily the case if $\Phi_{M,N}(x,c) = 0$ --- then $x \in \Zbar$ if and only if $c \in \Zbar$.

\begin{lem}\label{lem:options}
Let $x \in K$, and let $c_1,\ldots,c_n$ be the roots of $\Phi_{M,N}(x,C) \in K[C]$. Then one of the following must be true:
	\begin{enumerate}
		\item for all $i \in \{1,\ldots,n\}$, $x$ has preperiod equal to $M$ for $f_{d,c_i}$; or
		\item for all $i \in \{1,\ldots,n\}$, $x$ has eventual period equal to $N$ for $f_{d,c_i}$.
	\end{enumerate}
\end{lem}

\begin{proof}
Let $i \in \{1,\ldots,n\}$ be arbitrary. If $x \in \Zbar$, then $c_i \in \Zbar$, and therefore $f_{d,c_i}^M(x) \in \Zbar$. Since $f_{d,c_i}^M(x)$ has formal period $N$, Lemma~\ref{lem:smaller_period} implies that $f_{d,c_i}^M(x)$ must have \emph{exact} period $N$, and therefore $x$ has eventual period $N$ for $f_{d,c_i}$. On the other hand, if $x \not \in \Zbar$, then it follows from Lemma~\ref{lem:smaller_preperiod} that $x$ must have preperiod equal to $M$ for $f_{d,c_i}$.
\end{proof}

Now let $x \in K$ be such that $x$ does not realize portrait $(M,N)$ in degree $d$. It follows from Lemma~\ref{lem:options} that either $x$ has preperiod strictly less than $M$ for $f_{d,c}$ for every root $c$ of $\Phi_{M,N}(x,C)$, or $x$ has eventual period strictly less than $N$ for all such maps $f_{d,c}$. We handle these two cases separately.

\subsection{Eventual period less than formal eventual period}\label{sec:smaller_period}

Throughout this section, we suppose the tuple $(x,M,N,d) \in K \times \bbZ^3$, with $M \ge 0$, $N \ge 1$, and $d \ge 2$, satisfies the following condition:
\begin{equation}\label{eq:c1}
\tag{$*$}
\text{For all roots $c$ of $\Phi_{M,N}(x,C) \in K[C]$, $x$ has eventual period strictly less than $N$ for $f_{d,c}$}.
\end{equation}

Now fix one such root $c \in K$, and assume for the moment that $M \ge 1$. By Lemma~\ref{lem:factors}, $\zeta f_{d,c}^{M-1}(x)$ is periodic for some root of unity $\zeta \ne 1$. The period of $\zeta f_{d,c}^{M-1}(x)$ is equal to the period of $f_{d,c}(\zeta f_{d,c}^{M-1}(x)) = f_{d,c}^M(x)$, which is less than $N$ by \eqref{eq:c1}. Lemma~\ref{lem:smaller_period} then implies that
	\[
		\left. \frac{\partial \Phi_N(Z,C)}{\partial Z} \right|_{(\zeta f_{d,c}^{M-1}(x),c)} = 0.
	\]
Therefore, using the factorization appearing in Lemma~\ref{lem:factors} and applying the chain rule, we have
	\begin{equation}\label{eq:ddX}
		\left. \frac{\partial \Phi_{M,N}(X,C)}{\partial X} \right|_{(x,c)} = 0.
	\end{equation}

Note that if $M = 0$, then \eqref{eq:ddX} holds immediately by Lemma~\ref{lem:smaller_period}. In this case, since $Y_1(0,N)$ is nonsingular for all $N \ge 1$, we conclude that $\frac{\partial}{\partial C} \Phi_{0,N}(X,C)$ does not vanish at $(x,c)$. The same is true for $M \ge 1$: Indeed, by Lemma~\ref{lem:options}, $x$ must have preperiod equal to $M$ for $f_{d,c}$, and therefore $f_{d,c}^{M-1}(x) \ne 0$ by Lemma~\ref{lem:smaller_preperiod}. It then follows from Lemma~\ref{lem:factors} that $(x,c)$ is a nonsingular point on $Y_1(M,N)$, so the $C$-partial of $\Phi_{M,N}(X,C)$ cannot vanish at $(x,c)$.

In any case, we have shown that each root of $\Phi_{M,N}(x,C) \in K[C]$ is a \emph{simple} root, so the number of distinct roots of $\Phi_{M,N}(x,C)$ is precisely
	\[
		\deg_C \Phi_{M,N} =
			\begin{cases}
				\dfrac{1}{d}D(N), &\mbox{ if } M = 0;\\
				(d - 1)d^{M-2} D(N) , &\mbox{ if } M \ge 1.
			\end{cases}
	\]
On the other hand, since every root satisfies $\Phi_{M,n}(x,c) = 0$ for some $n$ strictly dividing $N$, the number of roots of $\Phi_{M,N}(x,C)$ can be at most
	\[
		\sum_{\substack{n \mid N \\ n < N}} \deg_C \Phi_{M,n} =
			\begin{cases}
			\dfrac{1}{d}\sum\limits_{\substack{n \mid N \\ n < N}} D(n) , &\mbox{ if } M = 0; \\
			(d-1) d^{M-2} \sum\limits_{\substack{n \mid N \\ n < N}} D(n) , &\mbox{ if } M \ge 1.
			\end{cases}
	\]
In particular, this means that
	\[
		D(N) \le \sum_{\substack{n \mid N \\ n < N}} D(n),
	\]
which implies that $N = d = 2$ by Lemma~\ref{lem:d(N)}.  We assume henceforth that $(N,d) = (2,2)$.

Suppose $M = 0$. In this case, \eqref{eq:c1} says that for every $c \in K$ with $\Phi_2(x,c)$ = 0 we also have $\Phi_1(x,c) = 0$. In the $d = 2$ case, we have
	\[
		\Phi_1(X,C) = X^2 - X + C \ , \ \Phi_2(X,C) = X^2 + X + C + 1.
	\]
The condition $\Phi_2(x,c) = \Phi_1(x,c) = 0$ implies that $(x,c) = (-1/2,-3/4)$. Therefore, if $x \ne -1/2$ and $\Phi_2(x,c) = 0$, then $x$ has exact period $2$ for $f_{2,c}$.

Now suppose $M = 1$, and let $x \in K$ with $x \ne 1/2$. By the previous paragraph, there exists $c \in K$ for which $\Phi_2(-x,c) = 0$ and $-x$ has period 2 under $f_{2,c}$. Since $d = 2$, Lemma~\ref{lem:factors} yields
	\[ \Phi_{1,2}(X,C) = \Phi_2(-X,C), \]
so for this particular value of $c$ we have $\Phi_{1,2}(x,c) = 0$. Moreover, since $f_{2,c}(x) = f_{2,c}(-x)$ has period 2, $x$ has eventual period 2 for $f_{2,c}$.

Finally, consider the case $M \ge 2$. Let $c \in K$ satisfy $\Phi_{M,2}(x,c) = 0$. By hypothesis, $x$ has portrait $(M,1)$ for $f_{2,c}$, which implies that
	\[ \Phi_2(f_{2,c}^M(x),c) = \Phi_1(f_{2,c}^M(x),c) = 0. \]
As explained above, this means that $c = -3/4$; in particular, the polynomial $\Phi_{M,2}(x,C)$ has only the single root $c = -3/4$. Since $\Phi_{M,2}(X,C)$ has degree $2^{M-2} D(2) \ge 2$ in $C$, the root $c = -3/4$ must be a multiple root of $\Phi_{M,2}(x,C)$, contradicting our previous assertion that $\Phi_{M,N}(x,C)$ has only simple roots.

We have shown that if $(x,M,N,d)$ satisfies \eqref{eq:c1}, then $(N,d) = (2,2)$ and $(x,M) \in \{(-1/2,0),(1/2,1)\}$. From this, we draw the following conclusion:

\begin{prop}\label{prop:small_period}
Let $(x,M,N,d) \in K \times \bbZ^3$ with $M \ge 0$, $N \ge 1$, and $d \ge 2$. Suppose that
	\[
		(x,M,N,d) \not \in \left\{\left(-\frac{1}{2},0,2,2\right), \left(\frac{1}{2},1,2,2\right) \right \}.
	\]
Then there exists $c \in K$ with $\Phi_{M,N}(x,c) = 0$ for which $x$ has eventual period equal to $N$ for $f_{d,c}$.
\end{prop}

If $(x,M,N,d)$ is any exception to Theorem~\ref{thm:main} not appearing in Proposition~\ref{prop:small_period}, then for every root $c$ of $\Phi_{M,N}(x,C)$, $x$ must have exact preperiod less than $M$ for $f_{d,c}$. We now consider this situation.

\subsection{Preperiod less than formal preperiod}\label{sec:smaller_preperiod}

Suppose now that $(x,M,N,d) \in K \times \bbZ^3$, with $M \ge 0$, $N \ge 1$, $d \ge 2$, satisfies the following condition:
\begin{equation}\label{eq:c2}
\tag{$**$}
\text{For all roots $c$ of $\Phi_{M,N}(x,C) \in K[C]$, $x$ has preperiod strictly less than $M$ for $f_{d,c}$}.
\end{equation}
For all such roots $c$, Lemma~\ref{lem:smaller_preperiod} implies that $f_{d,c}^{M-1}(x) = 0$ is periodic of period $N$, and therefore $x$ must have eventual period equal to $N$ for $f_{d,c}$.

If $M = 1$, then $f_{d,c}^{M-1}(x) = 0$ means precisely that $x = 0$, and we have already seen that $0$ cannot have portrait $(1,N)$ for $f_{d,c}$ for any $N \ge 1$ and $c \in K$. We will therefore assume that $M \ge 2$.

We first prove an elementary lemma.

\begin{lem}\label{lem:mult_root}
Suppose \eqref{eq:c2} is satisfied, and let $\zeta$ be a $d$th root of unity. If $M \ge 2$, then the polynomial $\Psi_{M,N}^{\zeta}(x,C) \in K[C]$ has a multiple root.
\end{lem}

\begin{proof}
Let $c$ be any root of $\Psi_{M,N}^{\zeta}(x,C)$. By Lemma~\ref{lem:factors}, this implies that $\Phi_{M,N}(x,c) = 0$, so we have $f_{d,c}^{M-1}(x) = 0$ by the assumption in \eqref{eq:c2}. Therefore $\Psi_{M,N}^{\zeta}(x,C)$ has at most
	\[
		\deg_C f_{d,C}^{M-1}(X) = d^{M-2}
	\]
distinct roots $c$. On the other hand, the degree (in $C$) of $\Psi_{M,N}^{\zeta}(x,C)$ satisfies
	\[
		\deg_C \Psi_{M,N}^{\zeta}(X,C) = \deg_C \Phi_N(\zeta f_{d,C}^{M-1}(X),C) = d^{M-2} D(N) > d^{M-2},
	\]
so $\Psi_{M,N}^{\zeta}(X,C)$ must have a multiple root.
\end{proof}

We now show that, in \emph{most} cases, if $f_{d,c}^{M-1}(x) = 0$ and $\Phi_{M,N}(x,c) = 0$, then $c$ must actually be a simple root of the polynomial $\Psi_{M,N}^{\zeta}(x,C)$ when $\zeta$ is a \emph{primitive} $d$th root of unity. Such cases contradict Lemma~\ref{lem:mult_root}, and therefore \eqref{eq:c2} must fail in these cases.

\begin{lem}\label{lem:simple_root}
Let $(M,N,d) \in \bbZ^3$ with $N \ge 1$; $M,d \ge 2$; and $(M,N,d) \ne (2,2,2)$. Let $\zeta$ be a primitive $d$th root of unity, and suppose $(x,c) \in K^2$ satisfies $\Psi_{M,N}^{\zeta}(x,c) = 0 = f_{d,c}^{M-1}(x)$. Then $c$ is a simple root of $\Psi_{M,N}^{\zeta}(x,C) \in K[C]$.
\end{lem}

\begin{rem}
Lemma~\ref{lem:simple_root} actually holds if $\zeta$ is any $d$th root of unity different from $1$, though the proof is somewhat more involved and we do not require this level of generality. We also note that $\zeta \ne 1$ is necessary: for example, if we take $x = 0$, $N = 1$, and let $d,M \ge 2$ be arbitrary, then $c = 0$ satisfies $f_{d,c}^{M-1}(0) = 0$, and one can check that $0$ is a multiple root of
	\[ \Psi_{M,1}^1(0,C) = \Phi_1(f_{d,C}^{M-1}(0),C) = \Phi_1(f_{d,C}^{M-2}(C),C) = \left( f_{d,C}^{M-2}(C) \right)^2 - f_{d,C}^{M-2}(C) + C. \]
\end{rem}

In order to prove Lemma~\ref{lem:simple_root}, we require the following description of the $C$-partials of the iterates of $f_{d,C}$. We omit the relatively simple proof by induction, but mention that the proof of the case $d = 2$ may be found in \cite[Lem. 3.3]{buff/lei:2014}.

\begin{lem}\label{lem:technical}
For $k \in \bbN$,
	\[
		\frac{\partial}{\partial C} f_{d,C}^k(X) = 1 + \sum_{j=1}^{k-1} d^j \cdot \prod_{i=1}^j f_{d,C}^{k-i}(X)^{d-1}.
	\]
\end{lem}

We also require the following special case of a result due to Morton and Silverman \cite[Thm. 1.1]{morton/silverman:1994}.  For a number field $F$, we will denote by $\OF$ the ring of integers of $F$.
\begin{lem}\label{lem:residue_period}
Let $F$ be a number field, and let $c \in \OF$. Let $p \in \bbZ$ be prime, let $\frakp \subset \OF$ be a prime ideal lying above $p$, and let $k_{\frakp} := \OF/\frakp$ be the residue field of $\frakp$. Suppose $P \in \OF$ has exact period $N$ for $f_{d,c}$, and suppose the reduction $\tilde{P} \in k_{\frakp}$ of $P$ has exact period $N'$ for $\tilde{f_{d,c}} \in k_{\frakp}[z]$. Then	\[
		N = N'  \text{ or } N = N'rp^e,
	\]
where $r$ is the multiplicative order of $\left(\tilde{f_{d,c}}^{N'}\right)'\left(\tilde{P}\right)$ in $k_{\frakp}$ and $e \in \bbZ_{\ge 0}$. In particular, if $\left(\tilde{f_{d,c}}^{N'} \right)'\left(\tilde{P}\right) = \tilde 0$, then $N = N'$.
\end{lem}

\begin{proof}[Proof of Lemma~\ref{lem:simple_root}]
Since $\Psi_{M,N}^\zeta (X,C) = \Phi_N(\zeta f_{d,C}^{M-1}(X),C)$, and since $\Phi_N(X,C)$ divides $f_{d,C}^N(X) - X$, it suffices to show that $c$ is a simple root of the polynomial
	\[
		f_{d,C}^N(\zeta f_{d,C}^{M-1}(x)) - \zeta f_{d,C}^{M-1}(x) = f_{d,C}^{M + N - 1}(x) - \zeta f_{d,C}^{M-1}(x),
	\]
which is equivalent to showing that
	\[
		\frac{d}{dC} \left(f_{d,C}^{M + N - 1}(x) - \zeta f_{d,C}^{M-1}(x)\right)
	\]
does not vanish at $C = c$. By Lemma~\ref{lem:technical}, this is equivalent to showing that
	\begin{equation}\label{eq:delta}
		\delta := 1 - \zeta  + \sum_{j=1}^{M + N - 2} d^j \cdot \prod_{i=1}^j f_{d,c}^{N + M - 1 - i}(x)^{d-1} - \zeta \cdot \sum_{j=1}^{M - 2} d^j \cdot \prod_{i=1}^j f_{d,c}^{M - 1 - i}(x)^{d-1}
	\end{equation}
is nonzero. The conditions of the lemma imply that $\Phi_N(0,c) = 0$, so that $c \in \Zbar$, and therefore the condition $f_{d,c}^{M-1}(x) = 0$ implies that $x \in \Zbar$ as well. Thus, $\delta$ is an algebraic integer; let $F := \bbQ(x,c,\zeta)$, so that $\delta \in \OF$.

Suppose first that $d$ is not a prime power. Then $1 - \zeta$ is an algebraic unit. Since $\delta = 1 - \zeta + d \alpha$ for some $\alpha \in \OF$, we have
	\[
		\delta \equiv 1 - \zeta \not \equiv 0 \mod d\OF.
	\]
In particular, $\delta \ne 0$.

Now suppose that $d = p^k$ is a prime power, in which case $1 - \zeta$ is no longer an algebraic unit. Let $\frakp \subset \OF$ be a prime ideal lying above $p \in \bbZ$. Then $\frakp \cap \bbZ[\zeta] = (1 - \zeta)$ and $p\bbZ[\zeta] = (1 - \zeta)^r$, where $r = \phi(d) = p^{k-1}(p-1)$. Therefore,
	\[
		\ord_{\frakp}(d) = k \cdot \ord_{\frakp}(p)  = kp^{k-1}(p-1) \cdot \ord_{\frakp}(1- \zeta),
	\]
which is strictly greater than $\ord_{\frakp}(1 - \zeta)$ unless $k = 1$ and $p = 2$; that is, unless $d = 2$.

If $\ord_{\frakp}(d) > \ord_{\frakp}(1 - \zeta)$, then we again write $\delta = 1 - \zeta + d \alpha$ for some $\alpha \in \OF$ and find that $\ord_{\frakp}(\delta) = \ord_{\frakp}(1 - \zeta)$ is finite, hence $\delta \ne 0$. For the remainder of the proof, we take $d = 2$ and, therefore, $\zeta = -1$. Observe that the second sum appearing in \eqref{eq:delta} is empty if $M = 2$. We therefore consider the cases $M = 2$ and $M > 2$ separately.

\textbf{Case 1:} $M > 2$. In this case, we have
	\begin{align*}
		\delta
		&= 2  + \sum_{j=1}^{M + N - 2} 2^j \cdot \prod_{i=1}^j f_{2,c}^{N + M - 1 - i}(x) + \sum_{j=1}^{M - 2} 2^j \cdot \prod_{i=1}^j f_{2,c}^{M - 1 - i}(x) \\
		&= 2\left(1 + f_{2,c}^{N+M-2}(x) + f_{2,c}^{M-2}(x) + 2\alpha \right)
	\end{align*}
for some $\alpha \in \OF$. To show that $\delta \ne 0$, it suffices to show that
	\[
	\beta := 1 + f_{2,c}^{N+M-2}(x) + f_{2,c}^{M-2}(x) \not \in 2\OF.
	\]
We are assuming that $f_{2,c}^{M-1}(x) = 0$ has period $N$ for $f_{2,c}$, so also $f_{2,c}^{N+M-1}(x) = 0$. Hence
	\[ f_{2,c}(f_{2,c}^{M-2}(x)) = 0 = f_{2,c}(f_{2,c}^{N+M-2}(x)). \]
Since $f_{2,c}^{M-2}(x)$ and $f_{2,c}^{N+M-2}(x)$ are preimages of a common point --- namely, $0$ --- under $f_{2,c}$, we must have
	\[ f_{2,c}^{N+M-2}(x) = \pm f_{2,c}^{M-2}(x). \]
This means that $\beta - 1 = f_{2,c}^{N+M-2}(x) + f_{2,c}^{M-2}(x) \in 2\OF$, and therefore $\beta \not \in 2\OF$, as desired.

\textbf{Case 2:} $M = 2$. Since the second sum appearing in \eqref{eq:delta} is empty, we may write
	\begin{align*}
		\delta = 2\Big( \left(1 + f_{2,c}^N(x)\right) + 2\alpha \Big)
	\end{align*}
for some $\alpha \in \OF$. Let $\frakp \subset \OF$ be any prime lying above 2, and let $k_{\frakp}$ denote the residue field of $\frakp$.

Now suppose that $\delta = 0$. We will show that we must have $N = 2$, which yields precisely the exception $(M,N,d) = (2,2,2)$ in the statement of the lemma and completes the proof.

Since $\delta = 0$, we must have $1 + f_{2,c}^N(x) \in \frakp$; that is, in $k_{\frakp}$ we have $\tilde{f_{2,c}^N(x)} = \tilde{-1}$. Since $f_{2,c}^{N+1}(x) = 0 = f_{2,c}(x)$ by hypothesis, we have
	\[
		\tilde{0} = \tilde{f_{2,c}^{N+1}(x)} = \left(\tilde{f_{2,c}^N(x)} \right)^2 + \tilde{c} = \tilde{1+c},
	\]
so $\tilde{c} = \tilde{-1}$. Therefore the period of $\tilde{0}$ under $\tilde{f_{2,c}}$ is equal to 2, since
	\[ \tilde{f_{2,-1}}(\tilde{0}) = \tilde{0^2 - 1} = \tilde{-1} \mbox{ and } \tilde{f_{2,-1}}(\tilde{-1}) = \tilde{(-1)^2 - 1} = \tilde{0}. \]
Since $\left(\tilde{f_{2,-1}^2}\right)'(\tilde{0}) = \tilde{0}$, it follows from Lemma~\ref{lem:residue_period} that $0$ must have period $N = 2$ for $f_{2,c}$, as claimed.
\end{proof}

Combining Lemmas~\ref{lem:mult_root} and \ref{lem:simple_root} yields the following:

\begin{prop}\label{prop:small_preperiod}
Let $(x,M,N,d) \in K \times \bbZ^3$ with $M,N \ge 1$ and $d \ge 2$. Suppose that
	\[
		(x,M) \ne (0,1) \text{ and } (x,M,N,d) \not \in \{(\pm 1, 2, 2, 2)\}.
	\]
Then there exists $c \in K$ with $\Phi_{M,N}(x,c) = 0$ for which $x$ has preperiod equal to $M$ for $f_{d,c}$.
\end{prop}

\begin{proof}
We prove the converse, so assume that there is no $c \in K$ satisfying $\Phi_{M,N}(x,c) = 0$ such that $x$ has preperiod equal to $M$ for $f_{d,c}$ --- that is, suppose $(x,M,N,d)$ satisfies condition \eqref{eq:c2}. We have already seen that if $M = 1$, then this assumption implies that $x = 0$.

For $M \ge 2$, it follows from Lemmas~\ref{lem:mult_root} and \ref{lem:simple_root} that $(M,N,d) = (2,2,2)$, so it remains only to show that $x \in \{\pm 1\}$. Let $c$ be a root of $\Phi_{2,2}(x,C)$. The sentence following \eqref{eq:c2} implies that
	\[ \Phi_2(0,c) = 0 = f_{2,c}(x). \]
Writing these expressions explicitly yields
	\[ c + 1 = 0 = x^2 + c, \]
and therefore $x = \pm 1$.
\end{proof}

\subsection{Proof of the main theorem}

We now combine the results of the previous sections to prove the main theorem.

\begin{proof}[Proof of Theorem~\ref{thm:main}]
Let $(x,M,N,d) \in K \times \bbZ^3$ with $M \ge 0$, $N \ge 1$, and $d \ge 2$. In the paragraphs immediately preceding the statement of Theorem~\ref{thm:main}, we verified that if $(x,M) = (0,1)$ or
	\[ (x,M,N,d) \in \left\{ \left(-\frac{1}{2},0,2,2 \right), \left(\frac{1}{2}, 1, 2, 2 \right), \left( \pm 1, 2, 2, 2 \right) \right\}, \]
then $x$ does not realize portrait $(M,N)$ in degree $d$.

Conversely, suppose $x$ does not realize portrait $(M,N)$ in degree $d$. By Lemma~\ref{lem:options}, this means that one of the following must be true:
\begin{enumerate}
\item for every root $c$ of $\Phi_{M,N}(x,C)$, $x$ has eventual period less than $N$; or
\item for every root $c$ of $\Phi_{M,N}(x,C)$, $x$ has preperiod less than $M$.
\end{enumerate}

If (A) is satisfied, then $(x,M,N,d) \in \{(-1/2,0,2,2), (1/2,1,2,2)\}$ by Proposition~\ref{prop:small_period}; if (B) is satisfied, then $(x,M) = (0,1)$ or $(x,M,N,d) \in \{(\pm 1, 2, 2, 2)\}$ by Proposition~\ref{prop:small_preperiod}.
\end{proof}

\section{Further questions}

One might ask the following more general question: Let $K$ be an algebraically closed field of characteristic zero, let $\calK := K(t)$ be the function field in one variable over $K$, and let $\phi_d(z) := z^d + t \in \calK[z]$. Let $(x,M,N,d) \in \calK \times \bbZ^3$ with $M \ge 0$, $N \ge 1$, and $d \ge 2$. Does there exist a prime $\frakp \in \Spec \calO_{\calK}$ such that, modulo $\frakp$, $\tilde{x}$ has portrait $(M,N)$ for $\tilde{\phi_d}$? Theorem~\ref{thm:main} answers this question when $x$ is chosen to be a constant point (i.e., $x \in K$), since reducing modulo a place of $\calK$ is equivalent to specializing $t$ to a particular element of $K$.

There are at least two tuples $(x,M,N,d)$ with $x \in \calK$ \emph{non}-constant for which the answer to the above question is negative: one can show that if
	\[
		(x,M,N,d) \in \{ (-t,1,1,2) , (t+1,1,2,2) \},
	\]
then there is no place $\frakp$ such that $\tilde{x}$ has portrait $(M,N)$ for $\tilde{\phi_d}$ (modulo $\frakp$). We do not know if there are any other such examples; however, it follows from the results of \cite{ghioca/nguyen/tucker:2015} that, for a fixed $d \ge 2$, the set of remaining examples is finite and effectively (though perhaps not \emph{practically}) computable.

Another direction one might pursue is to consider Question~\ref{ques:main} with $K$ an algebraically closed field of positive characteristic. In this case, the analogue of Baker's theorem (Theorem~\ref{thm:baker}) was proven by Pezda \cite{pezda:1994, pezda:1996, pezda:1998}. Pezda's theorem is more complicated than that of Baker, so it seems that a proof of the positive-characteristic analogue of Theorem~\ref{thm:main} would also be considerably more involved. Another obstacle is the fact that the polynomials $\Phi_N(X,C)$ are not generally irreducible in positive characteristic, so the methods of this article would require significant modifications if they are to be used to prove a version of the main theorem in positive characteristic.

\bibliography{C:/Dropbox/jdoyle}{}
\bibliographystyle{amsplain}

\end{document}